\documentclass[a4paper,10pt]{amsart}
\usepackage{amssymb}
\usepackage[utf8]{inputenc}
\usepackage[T1]{fontenc}                %uncomment these two lines
\usepackage[french,english]{babel}

\usepackage[hidelinks]{hyperref}

\usepackage{color}

\theoremstyle{plain}
\newtheorem{Lemma}{Lemma}

\newtheorem{Proposition}[Lemma]{Proposition}
\newtheorem*{Theorem*}{Theorem}

\theoremstyle{definition}
\newtheorem{Definition}[Lemma]{Definition}

\theoremstyle{remark}
\newtheorem{Remark}[Lemma]{Remark}

\newcommand{\R}{\mathbb{R}}     %real numbers
\newcommand{\HH}{\mathbb{H}}
\newcommand{\flatC}{\mathtt{fC}}
\newcommand{\fullC}{\mathtt{C}}

\newcommand{\disk}{\mathbb D}
\newcommand{\vcone}{\mathtt{vC}}

\title[Intrinsic rectifiability via flat cones in the Heisenberg group]{Intrinsic rectifiability via flat cones\\ in the Heisenberg group}

\author[Julia]{Antoine Julia}
	\address[A.~Julia]{Dipartimento di Matematica ``T.Levi-Civita'', via Trieste 63, 35121 Padova, Italy.}
	\email{ajulia@math.unipd.it}

\author[Nicolussi~Golo]{Sebastiano Nicolussi Golo}
	\address[S.~Nicolussi~Golo]{Dipartimento di Matematica ``T.Levi-Civita'', via Trieste 63, 35121 Padova, Italy.}
	\email{sebastiano2.72@gmail.com}

\thanks{
	Both A.J.~and~S.N.G.~have been supported
	by University of Padova STARS Project ``Sub-Riemannian Geometry and Geometric Measure Theory Issues: Old and New'';
	and by the INdAM – GNAMPA Project 2019 ``Rectifiability in Carnot groups''.}

\subjclass[2010]{%
	53C17,% sub-Riemannian geometry
	28A75,% Length, area, volume, other geometric measure theory
	22E25% Nilpotent and solvable Lie groups
	}
\keywords{%
	Sub-Riemannian Geometry, %
	Heisenberg Group, %
	Intrinsic Lipschitz graphs.
	}

\date{\today}

\begin{document}
\begin{abstract}
We give a geometric criterion for a topological surface in the first Heisenberg group to be an intrinsic Lipschitz graph, using planar cones instead of the usual open cones.
\end{abstract}
\maketitle
%\tableofcontents

%%%%%%%%%%%%%%%%%%%%%%%%%%%%%%%%%%%%%%%%%%%%%%%%%%%%%%%%%%%%%%%
\section*{Introduction}
We identify the \emph{first Heisenberg group} $\HH$ with the manifold $\R^3$ endowed with the group operation 
\[
(x,y,z) (x',y',z') = \left(x+x',y+y', z+z' + \frac{xy' - x'y}{2}\right).
\]
The inverse of $(x,y,z)$ is $(-x,-y,-z)$.
If $E\subset\HH$ and $p\in\HH$, we write $pE$ for the set $\{pe:e\in E\}$.
The one-parameter family of group automorphisms $\delta_\lambda:(x,y,z)\mapsto(\lambda x,\lambda y,\lambda^2 z)$
are called \emph{dilations}.

In recent years, 
intrinsic Lipschitz graphs in $\HH$ have gained more attention,
%it has become common to study intrinsic graphs in $\HH$, in particular intrinsic Lipschitz graphs, 
see \cite{MR2287539,MR3511465,MR3642743} and \cite{2018arXiv181013122F,2017arXiv170808444C,zbMATH06994781,2019arXiv191200493D}. Indeed, these graphs yield a robust notion of rectifiable set in the Heisenberg group, and more generally in Carnot groups. 
%This notion is based on the decomposition of $\HH$ as a semi-direct product.
The definition of intrinsic Lipschitz graphs is inspired by the euclidean characterization of Lipschitz graphs by cones. 
Here, a \emph{cone} is a set $E\subset\HH$ with $\delta_\lambda(E)=E$ for all $\lambda>0$.
A set is an intrinsic Lipschitz graph if it has the following full cone property.
%Note that $0\in\bar E$, the closure of $E$.
%A well established notion of rectifiability uses open cones with horizontal axis, \cite{MR3511465,MR2287539}:
\begin{Definition}[Full cone property]
	A set $S\subset \HH$ has the \emph{full cone property} %is an \emph{intrinsic Lipschitz graph} 
if there is an open cone $C$ with $-C=C$, $C\cap \{z=0\}\neq\emptyset$ and, for all $p\in S$, 
	\[
	C\cap p^{-1}S = \emptyset .
	\]
\end{Definition}
An example of open cone, for some $\alpha>0$, is the \emph{$\alpha$-full cone}:
\[
        \fullC(\alpha) 
        :=  \left \{(x,y,z):\ |y|<\alpha|x|,\ |z|< \alpha x^2/2 \right\}.
\]
In this paper, we consider
\emph{flat cones}
% another type of cone: For $\alpha>0$, define the \emph{flat cone}
	\[
	\flatC(\alpha) := \{(x,y,0):|y| < \alpha|x|\} . %= \fullC(\alpha)\cap \{z=0\}.
	\]
Our aim is to 
%study the following property:
compare the full cone property with a flat cone property defined as follows.
\begin{Definition}[Flat cone property]
	A set $S\subset\HH$ has the \emph{$\alpha$-flat cone property} if for all $p\in S$
	\[
	\flatC(\alpha)\cap p^{-1}S =  \emptyset.
	\]
\end{Definition}
Note that, %in the definition of Intrinsic Lipschitz graph, 
%the particular type of cone does not matter much, as 
if $C$ is an open cone with $-C=C$ and $C\cap \{z=0\}\neq\emptyset$, then there exists $\alpha>0$ such that $\fullC(\alpha)\subset C$, up to a rotation around the $z$ axis.
Recall that rotations around the $z$ axis are isomorphisms of the homogeneous structure of $\HH$ given by left-translations and dilations.
% there exists $\alpha>0$ such that $\flatC(\alpha)\subset C$. 
Noting that $\flatC(\alpha)=\fullC(\alpha)\cap \{z=0\}$, it is then clear that an intrinsic Lipschitz graphs also have the flat cone property, up to a rotation around the $z$-axis. 
We will prove the inverse implication for topological surfaces.
\begin{Theorem*}%\label{mainthm}
	If $S\subset\HH$ is a topological surface with the flat cone property,
	then it has locally the full cone property. 
	Quantitatively, the $\alpha$-flat cone property implies locally the $\alpha/4$-full cone property.
\end{Theorem*}
\begin{Remark}
	If $S$ is a topological surface with the flat cone property, then it is an intrinsic graph and characteristic lines on $S$ are Lipschitz with uniformly bounded slope. The Theorem %~\ref{mainthm}
 is thus a geometric version of the result of \cite{MR3400438} in the first Heisenberg group.
\end{Remark}
\begin{Remark}
	The hypothesis in the Theorem %~\ref{mainthm}
 that $S$ be a topological surface is crucial. 
 We register two examples where our strategy fails and that we consider prototypical.
 The first one has the flat cone property, but not the full cone property at $0$ in the direction of the $x$ axis:
%	As a counterexample, the following set has the flat cone property, but not the full cone property in any neighbourhood of $0$:
\[
 \{(x,0,x^{3}), x\in [\,0,1\,]\}.
\]
 Another example is the following set, which has the flat cone property, but is not an intrinsic Lipschitz graph:
        \begin{equation}\nonumber
          \{(0,0,0)\} \cup \{(x,y,z): x=1\}\backslash\{(1,s,0): |s|<1\}.
        \end{equation}
 However, both sets are clearly intrinsic $1$-codimensional-rectifiable in the sense of  \cite[Definition~3.16]{MR2287539}. We do not know if the flat cone property implies intrinsic rectifiability.
\end{Remark}
%\subsection*{Sketch of proof}
We present a detailed proof in the subsequent sections.
However, the expert reader could be satisfied with the following sketch.\\
Suppose $p=0\in S$: we want to find $\beta>0$ and $0\in U\subset \HH$ open such that $\fullC(\beta)\cap (S\cap U) = \emptyset$.
We start by considering the intersection $S\cap H$, where $H=\{z=0\}$ is the so called ``horizontal plane'':
for every $y_0$ close to $0$ there is $x\in[\,- y_0/\alpha, y_0/\alpha\,]$ such that $(x,y_0,0)\in S$.
Now, we notice a crucial fact: all the left-translations 
%of $H$ of the form $(x,y_0,0)H$
$\{(x,y_0,0)H\}_{x\in\R}$, with $y_0$ fixed, are affine planes all containing a common line, which turns out to be inside the plane $\{y=0\}$.
Therefore, 
\[
\bigcap_{x\in[\,- y_0/\alpha, y_0/\alpha\,]} (x,y_0,0)\flatC(\alpha)
\]
is the union of two half lines in the plane $\{y=0\}$.
By construction, these two half lines do not intersect $S$.

\noindent If we then consider all $y_0\in(-\epsilon,\epsilon)$ for some $\epsilon>0$ small enough,
we get a family of half lines in $\{y=0\}$ that do not intersect $S$:
the union of such a family yields a ``vertical set'' which, in a neighbourhood of $0$, coincides with $\fullC(\alpha/2)\cap\{y=0\}$.

\noindent Although promising,
this is still not enough to get a full cone.
However, the same argument can be carried out in another system of coordinates $(x',y',z')$ for $\HH$ in which the $x'$-axis is slightly tilted with respect to the $x$-axis;
in other words, we consider $x'=x$, $y'=y+tx$, $z'=z$, where $t\in\R$.
Indeed, for $t$ small, $S$ will still have the $(\alpha/2)$-flat cone property in the new coordinates and thus we obtain a vertical set in $\{y'=0\}$ that does not intersect $S$ and coincides in a neighbourhood of $0$ with $\fullC(\alpha/4)\cap\{y'=0\}$.

\noindent The union of all these vertical sets coincides with $\fullC(\alpha/4)$ near $0$, as desired. This is proved in Section \ref{sec:flat-full}.
Of course, the arguments should be carried out in a uniform way for $p$ close to $0$ in $S$, this is the scope of Section \ref{sec:intrinsic-graphs}.
The full proof of the Theorem is in Section~\ref{sec03131147}.

\subsection*{Acknowledgements}
We thank Davide~Vittone and Roberto~Monti for inspiring conversations.

%%%%%%%%%%%%%%%%%%%%%%%%%%%%%%%%%%%%%%%%%%%%%%%%%%%%%%%%%%%%%%%
\section{Flat cones build up full cones}\label{sec:flat-full}
In this section, we show how we get truncated full cones from flat cones.
It is worth mentioning \cite{MR2875838}, where another strategy to pass from flat cones to full open cones is discussed.

	For $r,\beta>0$, define the \emph{truncated full cone}
	\[
	\fullC(\beta,r)
	:= \left \{(x,y,z):\ |y|<\beta|x|,\ |z|< \beta x^2/2,\ |x|< r \right\}
	\]
	and the \emph{vertical set}
	\[
	\vcone(\beta,r) = \left\{\left( x , 0,u x^2/2 \right) 
		: \vert u \vert < \beta, 
                \vert x\vert <r \right\}.
	\]
%	Note that $\fullC(\beta,r) = \fullC(\beta)\cap\{\vert x\vert < r\}$.
%	Thus, if a set $S$ satisfies 
%	 $\fullC(\beta,r)\cap p^{-1}S = \emptyset$
%	 for all $p\in S$,
%	 then for $o\in S$ the set $S\cap o\{\vert x\vert < r/2\}$ has the $\beta$-full cone property.
\begin{Remark}\label{rem:truncated2local}
	Note that if a set $S$ satisfies $\fullC(\beta,r)\cap p^{-1}S = \emptyset$,
%	\[
%	\fullC(\beta,r)\cap p^{-1}S = \emptyset ,
%	\]
	for all $p\in S$, then, for each $o\in S$, the set $S\cap o\{\vert x\vert < r/2\}$ has the $\beta$-full cone property.
\end{Remark}
%        We also consider the the \emph{vertical set}
%	\[
%	\vcone(\beta,r) = \left\{\left( x , 0,u x^2/2 \right) 
%		: \vert u \vert < \beta, 
%                \vert x\vert <r \right\}.
%	\]
%        For $t\in\R$, define the map $M_t(x,y,z) := (x,y+tx,z)$,
%	which is a group automorphism $\HH\to\HH$, with inverse $M_{t}^{-1} = M_{-t}$. 
%%        The next Lemma is an immediate consequence of the above definitions:
%	An immediate consequence of the above definitions is that, for all $r, \beta>0$,
%	\begin{equation}\label{eq:vCtofullC}
%	\fullC(\beta,r)= \bigcup_{\vert t\vert <\beta} M_t(\vcone(\beta,r)).
%      \end{equation}
%\begin{Lemma}\label{lem02241817}
%        Given $r, \beta>0$, we have 
%	\begin{equation}\label{eq:vCtofullC}
%	\fullC(\beta,r)= \bigcup_{\vert t\vert <\beta} M_t(\vcone(\beta,r)).
%      \end{equation}
%    \end{Lemma}
% {\color{red}\begin{proof}\Anote{not needed}
%   Pick $(x,y,z) \in \fullC(\beta,r)$, there holds $0< \vert y/x \vert < \beta$, so
%       \[
%            M_{-y/x} (x,y,z) = (x,0,z)\subset \vcone(\beta,r).
%       \]
%   Conversely, one can check that given $(x,0,z) \in \vcone (\beta,r)$ and $s \in \,]-\beta,\beta[\,$, there holds $M_s(x,0,z) \in \fullC(\beta,r)$.
% \end{proof}
% }
\begin{Lemma}\label{lem02242245}
	Fix $\epsilon,\beta>0$, then
	\begin{equation}\label{eq02251449}
	 \vcone\left(\frac{\beta}{2},\frac{2\epsilon}{\beta} \right) \subset 
         \bigcup_{\vert \eta\vert < \epsilon\,}
         \bigcap_{\,\vert s\vert<\beta^{-1}} (s\eta,\eta,0) \flatC(\beta).
	\end{equation}
\end{Lemma}

\begin{proof}
Pick $(x,0,ux^2/2) \in \vcone(\beta/2,2\epsilon/\beta)$, and notice that $\vert -ux\vert < \epsilon$ . 
We will show that $(x,0,ux^2/2) \in (-sux,-ux,0) \flatC(\beta)$, 
i.~e.~that $(sux,ux,0)(x,0,ux^2/2)\in \flatC(\beta)$,
for all $s\in(-\beta^{-1},\beta^{-1})$. We have
\[
    (sux,ux,0)(x,0,ux^2/2)= (x+sux, ux,0),
\]
so we only need to prove that $\vert ux\vert < \beta\vert x+sux\vert $, which is clear as $\vert u\vert <\beta/2$, $\vert su\vert <1/2$ and thus
\[
\beta\vert x+sux\vert 
\ge \beta \vert x\vert (1-\vert su\vert)
\ge \frac\beta2 \vert x\vert
\ge \vert ux\vert . \qedhere
\]
\end{proof}
%{\color{red} Combining Lemma  \ref{lem02241817} and Lemma \ref{lem02242245} connects flat and full cones:
%\begin{equation}\nonumber%\label{eq:CinMsvC}
%     \fullC(\beta/2, 2\epsilon/\beta) \subset \bigcup_{\vert t\vert < \beta}\bigcup_{\,\vert y_0\vert <\epsilon\,} \bigcap_{\,\vert s\vert<\beta^{-1}} M_t\left((sy_0,y_0,0) \flatC(\beta)\right).
%\end{equation}}

For $t\in\R$, define the map $M_t(x,y,z) := (x,y+tx,z)$,
which is a group automorphism $\HH\to\HH$, with inverse $M_{t}^{-1} = M_{-t}$. 
%        The next Lemma is an immediate consequence of the above definitions:
An immediate consequence of the above definitions is that, for all $r, \beta>0$,
\begin{equation}\label{eq:vCtofullC}
\fullC(\beta,r)= \bigcup_{\vert t\vert <\beta} M_t(\vcone(\beta,r)).
\end{equation}

%%%%%%%%%%%%%%%%%%%%%%%%%%%%%%%%%%%%%%%%%%%%%%%%%%%%%%%%%%%%%%%
\section{Intrinsic graphs}\label{sec:intrinsic-graphs}
Given $\Omega\subset\R^2$ and $\phi:\Omega\to\R$, the \emph{intrinsic graph} of $\phi$ is the subset %of $\HH$ given by
\begin{align*}
\Gamma_\phi :=& \{ (0,\eta,\tau)(\phi(\eta,\tau),0,0) : (\eta,\tau)\in\Omega \}\\
=& \{ (\phi(\eta,\tau),\eta,\tau -\eta\phi(\eta,\tau)/2) : (\eta,\tau)\in\Omega \}.
\end{align*}
Define the projection $\pi:\HH\to\R^2$, $(x,y,z) \mapsto( y,z+xy/2 )$. 
Notice that if $(\eta_0,\tau_0)=\pi(x_0,y_0,z_0)$, then $(x_0,y_0,z_0) = (0,\eta_0,\tau_0)(x_0,0,0)$.
%In other words, $\{(0,\eta_0,\tau_0)\} = \{(x_0,y_0,z_0)(\xi,0,0):\xi\in\R\}  \cap \{x=0\}$.\Anote{I'm not sure the last sentences are useful}
\begin{Proposition}\label{prop02251533}
	If $S$ is a topological surface with the flat cone property,
	then it is the intrinsic graph of a continuous function defined on an open subset of $\R^2$.
\end{Proposition}
\begin{proof}
	Since $\{(\xi,0,0):\xi\in\R\}\subset\flatC(\alpha)$, the map $\pi|_S:S\to\R^2$ is injective.
	Thus, $S=\Gamma_\phi$ for some $\phi:\pi(S)\to\R$.
	As $S$ is a topological surface and by the Invariance of Domain Theorem, $\pi(S)$ is open and $\phi$ is continuous.
\end{proof}

As we explained in the introduction, % sketch of the proof of the Theorem, 
in order to apply Lemma~\ref{lem02242245},
we will use the fact that for every $y_0$ close to $0$ there is $x\in[\,- y_0/\alpha, y_0/\alpha\,]$ such that $(x,y_0,0)\in S$:
\begin{Proposition}\label{prop02221801}
	Fix $\eta_0,\tau_0,\beta>0$ and let $\psi:[\,-\eta_0,\eta_0\,]\times[\,-\tau_0,\tau_0\,]\to\R$ be continuous, such that $\psi(0)=0$ and $\Gamma_\psi$ has the $\beta$-flat cone property.
	Then, letting $\epsilon:=\min\{\eta_0,\sqrt{\tau_0 \beta}\}$, for all $\eta\in [\,-\epsilon,\epsilon\,]$, there exists $s\in [\,-\beta^{-1},\beta^{-1}\,]$ with
	\begin{equation}\nonumber%\label{eq:xexists}
	    (s\eta,\eta,0) \in \Gamma_\psi.
	\end{equation}
\end{Proposition}
\begin{proof}
        Consider the function $\zeta(\eta,\tau):= \tau - \frac{\eta \psi(\eta,\tau)}{2}$ (yielding the third coordinate of the point $(0,\eta,\tau)(\psi(\eta,\tau),0,0)$). 
	Notice that, since $\flatC(\beta)\cap \Gamma_\psi =\emptyset$, we have
        \[
        \{(\eta,\tau): \zeta(\eta,\tau)=0\}\subset\R^2\setminus\pi(\flatC(\beta)) = \{(\eta,\tau) : \eta^2 \ge 2\beta \vert \tau\vert\},
        \]
	Moreover, by the choice of $\epsilon$, for $\eta \in [\,-\epsilon,\epsilon\,]$, there holds $\eta^2 < 2\beta \tau_0$ and therefore $\zeta(\eta,\pm \tau_0)\neq 0$. Since $\zeta$ is continuous,
 and $\zeta(0,\tau_0)=\tau_0>0>-\tau_0 =\zeta(0,-\tau_0)$,  we have for all $\eta\in [\,-\epsilon,\epsilon\,]$
	\begin{equation}
	\zeta(\eta,\tau_0) > 0  >\zeta(\eta,-\tau_0).
	\end{equation}
	Using once again the continuity of $\zeta$, it follows that for every $\eta\in[-\epsilon,\epsilon]$ there is $\tau\in]\,-\tau_0,\tau_0\,[$ with $\zeta(\eta,\tau)=0$ and thus a point $(s\eta,\eta,0) \in \Gamma_\psi$.
%         \begin{equation*}
%              (0,\eta,\tau)(\psi(\eta,\tau),0,0) = (\psi(\eta,\tau),\eta,0) \in \Gamma_\psi.
%         \end{equation*}
% %	In conclusion, we have proven that there exists $s\in \R$ such that $(s\eta,\eta,0)\Gamma_\psi$.
By the $\beta$-flat cone property of $\Gamma_\psi$, either $\eta=0$ (and we can take $s=0$) or there holds $\vert s \vert \leq \beta^{-1}$. % and the claim is proved.
\end{proof}

We want to apply Proposition \ref{prop02221801} and Lemma \ref{lem02242245}, not only to $S$ but also to $M_t(p^{-1}S)$ for $p$ in a neighborhood of $0$ and $t$ in a compact interval.
Lemma \ref{lem02251602bis} allows us to do this.
%
% after the changes of variables $M_t$ in order to use Proposition \ref{lem02241817} and translations in a neighbourhood of $0$ in $S$. And we need to do this in a uniform way. This is the scope of Lemma \ref{lem02251602}.

\begin{Remark}\label{rem03111937}
Notice that, for $t\in [\,- \alpha/2, \alpha/2\,]$, there holds 
$\flatC(\alpha/2)\subset M_t \flatC(\alpha)$.
In particular, if $S$ has the $\alpha$-flat cone property, then for such $t$, $M_t(p^{-1}S)$ has the $\alpha/2$-flat cone property, 
for every $t\in [\,- \alpha/2, \alpha/2\,]$ and $p\in\HH$.
%
% Therefore, the conclusions of Proposition \ref{prop02251533} apply to $M_t(S)$ and clearly also to $M_t(p^{-1}S)$ for $p\in S$. 
%\Anote{Seba, you wanted to rewrite this?} 
\end{Remark}

%Denote by $\disk_r$ the disk of radius $r$ in $\R^2$. For a fixed $r_0>0$ small enough, consider the map
%\[
%     \Phi: \bar{\disk}_{r_0}\to S, \quad \Phi(\eta,\tau) = (0,\eta,\tau)(\phi(\eta,\tau),0,0),
%\]
% which is a homeomorphism onto its image. Consider the following family of homeomorphisms:
%\[
%f_{t,p}:\bar{\disk}_{r_0}\to\R^2,
%\qquad
%f_{t,p}(\cdot)=\pi(M_t(p^{-1}\Phi(\cdot)) .
%\]\Anote{$\phi_{t,p}$ instead of $f_{t,p}$? But then we need $\Phi_{t,p}$}

\begin{Lemma}\label{lem02251602bis}
	Let $S\subset\HH$ be a topological surface with the $\alpha$-cone property and $0\in S$.
	Then there are $\eta_0,\tau_0>0$
	such that, defining $V_0:=\pi^{-1}([\,-\eta_0, \eta_0\,] \times [\,-\tau_0,\tau_0\,])$,
	for all $t\in [\,-\alpha/2,\alpha/2\,]$ and $p\in S\cap V_0$,
	there exists a continuous function $\phi_{t,p}: [\,-\eta_0, \eta_0\,] \times [\,-\tau_0,\tau_0\,]\to\R$
	such that 
	\[
	    \Gamma_{\phi_{t,p}} = M_t(p^{-1}S) \cap V_0 .
	\]
\end{Lemma}
\begin{proof}
	By Remark~\ref{rem03111937} and Proposition~\ref{prop02251533}, for all $p\in S$ and $t\in[\,-\alpha/2,\alpha/2\,]$,
	$ M_t(p^{-1}S)$ is an intrinsic graph over $\pi(M_t(p^{-1}S))$.
	 
	Denote by $\disk_r$ the disk of radius $r$ in $\R^2$. 
	Let $\Phi:\disk_1\to S$ be a local chart with $\Phi(0,0)=(0,0,0)$.
	For $p\in \Phi(\disk_1)$ and $t\in [\,-\alpha/2,\alpha/2\,]$, define
	\[
	f_{t,p}:\disk_1\to\R^2,
	\qquad
	f_{t,p}(v)=\pi(M_t(p^{-1}\Phi(v)) .
	\]
	Notice that, since $M_t(p^{-1}S)$ is an intrinsic graph, the $f_{t,p}$ are homeomorphisms onto their images.
	Define
	\begin{equation}\label{eq:omega}
		\Omega := \bigcap_{p\in \Phi(\bar{\disk}_{1/2})} \bigcap_{\,\vert t\vert \leq  \alpha/2} f_{t,p}({\disk}_{1}).
   \end{equation}
	
We claim that $\Omega$ is a neighborhood of $(0,0)$.
	We argue by contradiction. 
	Suppose that there exists a sequence  $g_n\to 0$ in $\R^2$, along with $p_n$ in $\Phi(\bar{\disk}_{1/2})$ and $t_n \in [\,-\alpha/2 ,\alpha/2\,]$, with $g_n\notin f_{t_n,p_n}({\disk}_{1})$. By compactness, we can suppose that  $p_n\to p_\infty\in \Phi(\bar{\disk}_{1/2})$ and that $t_n\to t_\infty \in [\,-\alpha/2,\alpha/2\,]$.
%As $f_{p_\infty,t_\infty}(\bar{\disk}_{r_0/2})$ is a neighbourhood of $0$ in $\R^2$, it contains the points $g_n$ for $n$ large enough. Note also that the maps $f_{t_n,p_n}$ converge uniformly to $f_{0,t_\infty}$. Thus, 
	Since the $f_{t,p}$ are homeomorphisms onto their images, there is $r>0$ such that $\disk_{2r}\subset f_{t_\infty,p_\infty}( \disk_{1})$.
	Since $f_{t_n,p_n}\to f_{t_\infty,p_\infty}$ uniformly, there holds $f_{t_n,p_n}(\partial \disk_{1})\cap\disk_r=\emptyset$ for $n$ large enough.
	In particular, as $0\in f_{t_n,p_n}(\disk_1)$, we have $\disk_r\subset f_{t_n,p_n}( \disk_{r_0})$, however, as $g_n\to 0$, for $n$ large enough there also holds $g_n\in\disk_r$, a contradiction.
	This proves the claim.
	
	Finally, if $\eta_0,\tau_0>0$ are such that $[\,-\eta_0, \eta_0\,] \times [\,-\tau_0,\tau_0\,]\subset\Omega$, then we have all the properties stated in the lemma.
\end{proof}

\section{Proof of the Theorem}\label{sec03131147} % \ref{mainthm}}
%\begin{proof}[Proof of Theorem \ref{mainthm}] 
Let $\alpha>0$ and $S\subset\HH$ be a topological surface with the $\alpha$-flat cone property.
We will show that $S$ has locally the $(\alpha/4)$-full cone property, that is, that for every $o\in S$ there is an open neighborhood $U\subset\HH$ of $o$ such that, for all $p\in S\cap U$
\[
   \fullC(\alpha/4)\cap p^{-1}(S\cap U) = \emptyset.
\]
It suffices to work for $o=0$.

By Lemma~\ref{lem02251602bis}, there are $\eta_0,\tau_0>0$
such that, defining $V_0:=\pi^{-1}([\,-\eta_0, \eta_0\,] \times [\,-\tau_0,\tau_0\,])$,
for all $t\in [\,-\alpha/2,\alpha/2\,]$ and $p\in S\cap V_0$,
there exists a continuous function $\phi_{t,p}: [\,-\eta_0, \eta_0\,] \times [\,-\tau_0,\tau_0\,]\to\R$
such that 
\[
    \Gamma_{\phi_{t,p}} = M_t(p^{-1}S) \cap V_0 .
\]
%where $M_t(p^{-1}S)$ has the $(\alpha/2)$-flat cone property
%by Remark~\ref{rem03111937}.
%\vspace{1cm}
%
%By Proposition \ref{prop02251533}, $S$ contains the intrinsic graph of a continuous function $\phi$ over an open disk $\disk_{r_0}\subset \R^2$, with $\phi(0)= 0$. Using Lemma \ref{lem02251602}, we know that there are $\eta_0,\tau_0>0$ such that for all $t\in [\,-\alpha/2,\alpha/2\,]$ and $p\in \Gamma_\phi$, there exists a continuous function $\phi_{t,p}: [\,-\eta_0, \eta_0\,] \times [\,-\tau_0,\tau_0\,]$  such that 
%\[
%    \Gamma_{\phi_{t,p}} \subset M_t(p^{-1}S),
%\]\Anote{Is there a straightforward way to define $\phi_{t,p}$ or should we use graph maps?}
%where $M_t(p^{-1}S)$ has the $(\alpha/2)$-flat cone property.
%we can choose $\Omega$ so that $M_t(p^{-1}S)\cap \pi^{-1}(\Omega)$ is the intrinsic graph of a continuous function $\phi_{t,p}:\Omega\to \R$.
% By Proposition \ref{prop02251533}, up to restricting to a neighbourhood of $0$, $S$ is the intrinsic graph of a continuous function $\phi:\Omega\to \R$, where $\Omega\subset \R^2$ is an open set. Let $\Phi$ and $\Omega$ be defined as in Lemma \ref{lem02251602} 
% and the preceding discussion. For $t\in \,]-\alpha/2,\alpha/2[\,$ and $p\in \Phi(\disk_{r_0/2})$, $M_t(p^{-1}S)$ contains the intrinsic graph of the continuous function $\phi_{t,p}: \Omega\to \R$, where $\phi_{t,p}$ is the inverse of $\pi\vert_S$ on $M_t$. Note that there exist $\eta_0$ and $\tau_0$ such that $[\,-\eta_0,\eta_0\,] \times [\,-\tau_0,\tau_0\,]\subset \Omega$.

Fix $p\in S\cap V_0$, $t\in [\,-\alpha/2,\alpha/2\,]$ and set $\epsilon:=\min\{\eta_0,\sqrt{\tau_0 \alpha/2}\}$.
Applying Proposition \ref{prop02221801} to $\phi_{t,p}$, we obtain for each $\eta\in [\,-\epsilon,\epsilon\,]$ a number $s\in [\,-2/\alpha,2/\alpha\,]$ such that $(s\eta,\eta,0)\in \Gamma_{\phi_{t,p}}$.
%Therefore $(s\eta,\eta,0) \in M_t(p^{-1}S)$ and, as $M_t(p^{-1}S)$ has the $\alpha/2$ flat cone property, there holds
%\[
%    M_t(p^{-1}S) \cap (s\eta,\eta,0)\flatC(\alpha/2).
%\]
Therefore, as $\Gamma_{\phi_{t,p}}\subset M_t(p^{-1}S)$ and since $M_t(p^{-1}S)$ has the $(\alpha/2)$-flat cone property by Remark~\ref{rem03111937}, there holds
\[
	M_t(p^{-1}S) \cap \bigcap_{s\in [\,-2/\alpha,2/\alpha\,]} (s\eta,\eta,0)\flatC(\alpha/2) = \emptyset ,
	\qquad \forall\eta\in [\,-\epsilon,\epsilon\,] .
\]
Taking the union over $\eta\in [\,-\epsilon,\epsilon\,]$ and
applying Lemma~\ref{lem02242245} 
%with $\beta:= \alpha/2$ and $\epsilon$ as above 
yields
\begin{equation}\label{eq03111906}
   M_t(p^{-1}S)\cap \vcone(\alpha/4,4\epsilon/\alpha) = \emptyset,
\end{equation}
which holds for all $p\in S\cap V_0$ and $t\in [\,-\alpha/2,\alpha/2\,]$.
Apply $M_{-t}$ and the left translation by $p$ to  \eqref{eq03111906},
then take a union over $t\in [\,-\alpha/2,\alpha/2\,]$ and use \eqref{eq:vCtofullC} to get
%We do this for each $t\in \,]-\alpha/2,\alpha/2[\,$ and use Lemma \ref{lem02241817} to get
\[
     S \cap p\,\fullC(\alpha/4,4\epsilon/\alpha)= \emptyset ,
\]
which holds for all $p\in S\cap V_0$.
Remark \ref{rem:truncated2local} implies that $S\cap V_0$ has the $(\alpha/4)$-cone property in a neighborhood of $0$.
Thus, $S$ has locally the $(\alpha/4)$-cone property.
%
%Doing the same thing for all $p$ in $\Phi([\,-\eta_0, \eta_0\,] \times [\,-\tau_0,\tau_0\,])$, and using Remark \ref{rem:truncated2local} yields the $\alpha/4$-cone property locally.
%\end{proof}
\qed

\bibliographystyle{abbrv}
\bibliography{RefsCarnot}

\end{document}